\documentclass{amsart}
\theoremstyle{plain}
\newtheorem{thm}{Theorem}

\newtheorem{lem}[thm]{Lemma}
\newtheorem{cor}[thm]{Corollary}
\theoremstyle{remark}
 \newtheorem{rem}[thm]{Remark}

\errorcontextlines=0 \numberwithin{equation}{section}

\begin{document}

\title[On a length preserving curve flow]{On a length preserving curve flow}

\author{ Li Ma, Anqiang Zhu}

\address{Li Ma, Department of Mathematical Sciences, Tsinghua University,
Beijing 100084, P. R. China}

\email{lma@math.tsinghua.edu.cn}

\dedicatory{}
\date{}

\thanks{The research is partially supported by the National Natural Science
Foundation of China 10631020 and SRFDP 20060003002}

\begin{abstract}
In this paper, we consider a new length preserving curve flow for
convex curves in the plane. We show that the global flow exists, the
area of the region bounded by the evolving curve is increasing, and
the evolving curve converges to the circle in $C^{\infty}$ topology
as $t\rightarrow \infty$.
\end{abstract}

\maketitle
%end topmatter

\section{Introduction}
In this paper, we study a nature evolution non-local flow for convex
curves in the plane. This flow preserves the length of curves. We
shall obtain the entropy estimate and integral estimates for the
evolution to get a global flow. We remark that this kind of
technique was used in \cite{GH}, where they studied the curve
shortening flow. Curve shortening flow has been studied extensively
in the last few decades (see \cite{GH} and \cite{MC} for background
and more references). It can be showed that the convexity of curves
along the curve shortening flow is preserved and the curves become
more and more circular before they collapse to a point. Since then,
other flows for curves have also been proposed. One may see
B.Andrews's papers (see for example, \cite{BA}) and Tsai's papers
\cite{T} for other kind of flows for curves. One may see \cite{H2}
for higher dimension flows for hyper-surfaces. As showed by M.Gage
\cite{G2}, some non-local flow for convex curves are also very
interesting. In a very recent paper \cite{PY}, S.L.Pan and J.N.Yang
consider a very interesting length preserving curve flow for convex
curves in the plane of the form $$ \frac{\partial}{\partial
t}\gamma(t)=(\frac{L}{2\pi}-k^{-1})N,
$$
where $L$, $N$, and $k$ are the length, unit normal vector,and the
curvature  of the curve $\gamma(t)$ respectively.
 They have proved that the convex plane curve will become more and
more circular and converges to circle in the $c^{\infty}$ sense.  It
is interesting to study curve flow which preserve some geometry
quantity, such as the area of the region bounded by the curve. For
this, one may see \cite{MC2} for a recent study.

The main result of this paper is the following theorem.

\begin{thm}\label{main}
Suppose $\gamma(u,0)$ is a convex curve in the plane $R^{2}.$ Assume
$\gamma(t):=\gamma(u,t)$ satisfies the following evolving equation
\begin{eqnarray}
\label{cf} \frac{\partial}{\partial t}\gamma(t)=(k-\alpha(t))N,
\end{eqnarray}
where $k$ is the curvature of the curve $\gamma(t)$ and
$$
\alpha(t)=\frac{1}{2\pi}\int k^{2}ds.
$$
Then the flow (\ref{cf}) is a length preserving flow. Furthermore,
we have the global flow $\gamma(t)$ and $\gamma(t)$ converges in
$C^{\infty}$ to the circle of of the fixed length $L$, as
$t\rightarrow \infty$.
\end{thm}

We remark that for $\alpha(t)=\frac{2\pi}{L}$, where $L$ is the
length of the curve $\gamma(t)$, the evolution equation (\ref{cf})
is area-preserving flow, which was studied by M.Gage in \cite{G2}.

The paper is organized as follows.

In section \ref{sect1}, we introduce necessary formulae for the flow
{\ref{cf}). We obtain key estimates about the curvature of the curve
flow (\ref{cf}) in section \ref{sect2}. We can justify the
assumption that the curve is a convex curve and the flow does not
blow up in finite time. We obtain the theorem \ref{main} in the
$C^{0}$ case. In the last section, we show the $C^{\infty}$
convergence of the flow.

\section{Preparation}\label{sect1}
First of all,  we derive basic formulae for our curve flow
(\ref{cf}).
\begin{lem} Let $w=|\gamma_u|$. Then we have
$$
w_t=-k(k-\alpha(t))w,
$$ and
$$
\frac{\partial}{\partial t} \frac{\partial}{\partial
s}-\frac{\partial}{\partial s} \frac{\partial}{\partial
t}=k(k-\alpha).
$$
\end{lem}
\begin{proof} Note that
$$
w^2=|\gamma_u|^2.
$$
Then we have
$$
ww_t=<\gamma_u,\gamma_{tu}>=<\gamma_u,((k-\alpha)N)_{u}>=w^{2}(k-\alpha)<T,N_s>.
$$
Using
$$
N_s=-kT,
$$
we have
$$
w_t=-k(k-\alpha)w.
$$
So,
\begin{eqnarray*}
\frac{\partial}{\partial t}\frac{\partial}{\partial
s}-\frac{\partial}{\partial s}\frac{\partial}{\partial
t}&=&\frac{\partial}{\partial
t}(\frac{1}{w})\frac{\partial}{\partial u}\\
&=&k(k-\alpha)\partial_{s}.
\end{eqnarray*}
\end{proof}
Recall that $ds=wdu$. Then we have
$$
(ds)_t=w_tdu=-k(k-\alpha)ds.
$$
We shall use this formula later.
\begin{lem}
$$
\frac{\partial}{\partial t}T=\partial_{s}kN.
$$
\end{lem}
\begin{proof}
\begin{eqnarray*}
\frac{\partial}{\partial t}T&=&\frac{\partial}{\partial
t}\frac{\partial}{\partial s}\gamma=\frac{\partial}{\partial
s}\frac{\partial}{\partial t}\gamma+k(k-\alpha)\partial_{s}\gamma\\
&=&\partial_{s}((k-\alpha)N)+k(k-\alpha)T\\
&=&\partial_{s}k N.
\end{eqnarray*}
\end{proof}

We denote the angle between the tangent and the X-axis by $\theta$.
Then we have
$$
cos\theta=<T,X>
$$
and
$$
k=\frac{\partial \theta}{\partial s}.
$$
\begin{lem}
$$
\frac{\partial\theta}{\partial t}=\partial_{s}k.
$$
\end{lem}
\begin{proof}
$$
-sin\theta \frac{\partial \theta}{\partial t}=<\frac{\partial
T}{\partial t},X>=\partial_{s}k<N,X>.
$$
\end{proof}

Then we can derive the important evolution equation for curvatures.
\begin{lem}
$$
\frac{\partial}{\partial t}k=\partial_{s}^{2}k+k^{2}(k-\alpha).
$$
\end{lem}
\begin{proof}
$$
\frac{\partial}{\partial t}\frac{\partial}{\partial
s}\theta=\frac{\partial}{\partial s}\frac{\partial }{\partial
t}\theta+k(k-\alpha)\partial_{s}\theta=\partial_{s}^{2}k+k^{2}(k-\alpha).
$$
\end{proof}

Hence, we have

\begin{lem}
$$
\frac{\partial}{\partial t}\int k ds=0.
$$
\end{lem}
\begin{proof}
$$
\frac{\partial}{\partial t}\int k ds=\int \frac{\partial}{\partial
t}k ds+ k \frac{\partial}{\partial t}ds=\int
\partial_{s}^{2}k+k^{2}(k-\alpha)- k^{2}(k-\alpha)ds=0.
$$
\end{proof}

Using this, we now derive
\begin{rem}
Since $\gamma(0)$ is a convex curve in the plane, we have $\int
k ds=2\pi.$
\end{rem}
Hence, we have
\begin{lem}
The length of the curve is fixed under the flow.
\end{lem}
\begin{proof}
\begin{eqnarray*}
\frac{\partial}{\partial
t}L&=&\int\frac{<\gamma_{u},\gamma_{tu}>}{|\gamma_{u}|}du=\int<T,\gamma_{ts}>ds\\
&=&\int<T,((k-\alpha)N)_{s}>ds=-\int k(k-\alpha)ds\\
&=&-\int k^{2} ds +\alpha\int k ds=2\pi \alpha-\int k^{2} ds=0.
\end{eqnarray*}
\end{proof}
Another important fact for us is the following
\begin{lem}
The area of the disk bounded by the curve $\gamma(t)$ is increasing.
That is,
$$
\frac{d}{dt}A(t)=\alpha L-2\pi\geq 0.
$$
\end{lem}
\begin{proof}
Since
$$
-2A(t)=\int <\gamma,N>ds,
$$
we have
\begin{eqnarray*}
-2\frac{d}{dt}A(t)&=&\int \frac{d}{dt}<\gamma,N> ds+\int
<\gamma,N>\frac{d}{dt}ds\\
&=&\int <(k-\alpha)N,N>ds+\int
<\gamma,-\partial_{s}(k-\alpha)T>ds\\
& &+\int
<\gamma,N>(-k(k-\alpha))ds\\
&=&\int (k-\alpha)ds-\int \partial_{s}(k-\alpha)<\gamma,T>ds-\int
<\gamma,N>k(k-\alpha)ds\\
&=&\int (k-\alpha)ds+\int (k-\alpha)(<T,T>+k<\gamma,N>)ds\\
& &-\int
<\gamma,N>k(k-\alpha)ds\\
&=&2\int (k-\alpha)ds=2(2\pi-L\alpha)
\end{eqnarray*}
Since
$$
2\pi\alpha=\int k^{2}ds\geq\frac{1}{L}(\int
kds)^{2}=\frac{4\pi^{2}}{L}.
$$
So
$$
2\pi-\alpha L\leq 0,
$$
which implies the result wanted.
\end{proof}

We remark that the above formula takes the equality when $k$ is a
constant, i.e. the curve is a circle.

We now consider the growth of the support function $$ P:=<X,N>
$$ and the growth of $\alpha$.

\begin{lem}
$$
\partial_{t}P=\partial_{s}^{2}P+2k-\alpha+k^{2}P.
$$
$$
\partial_{t}\alpha=-\frac{1}{\pi}\int
(k_{s})^{2}ds+\frac{1}{2\pi}\int k^{3}(k-\alpha)ds.
$$
$$
\partial_{t}(\frac{P}{k})=\partial_{s}^{2}(\frac{P}{k})+(P-\frac{1}{k})\alpha+2\frac{k_{s}}{k}\partial_{s}(\frac{P}{k}).
$$
\end{lem}
\begin{proof}

\begin{eqnarray*}
\partial_{t}P&=&\partial_{t}<\gamma,N>=<\partial_{t}\gamma,N>+<\gamma,\partial_{t}N>\\
&=&k-\alpha-k_{s}<\gamma,T>.
\end{eqnarray*}

\begin{eqnarray*}
\partial_{s}P&=&\partial_{s}<\gamma,N>=-k<\gamma,T>.
\end{eqnarray*}

\begin{eqnarray*}
\partial_{s}^{2}P&=&-k_{s}<\gamma,T>-k-k^{2}<\gamma,N>.
\end{eqnarray*}
Combine the above computation, we have
$$
\partial_{t}P=\partial_{s}^{2}P+2k-\alpha+k^{2}P.
$$

$$
\partial_{t}(\frac{k}{P})=\frac{k_{t}P-kP_{t}}{P^{2}}.
$$

$$
\partial{s}(\frac{k}{p})=\frac{k_{s}P-kP_{s}}{P^{2}}.
$$

$$
\partial_{s}^{2}(\frac{k}{P})=\frac{k_{ss}P-kP_{ss}}{P^{2}}-2\frac{P_{s}}{P}\partial_{s}(\frac{k}{P}).
$$

Combine the above computation, we have
\begin{eqnarray*}
\partial_{t}(\frac{k}{P})-\partial_{s}^{2}(\frac{k}{P})&=&\frac{(k_{t}-k_{ss})P
-k(P_{t}-P_{ss})}{P^{2}}+2\frac{P_{s}}{P}\partial_{s}(\frac{k}{P})\\
&=&\frac{k^{2}(k-\alpha)P-k(2k-\alpha+k^{2}P)}{P^{2}}+2\frac{P_{s}}{P}\partial_{s}(\frac{k}{P})\\
&=&\frac{\alpha(1-kP)-2k^{2}}{P^{2}}+2\frac{P_{s}}{P}\partial_{s}(\frac{k}{P}).
\end{eqnarray*}

\begin{eqnarray*}
\partial_{t}\alpha &=&\frac{1}{\pi}\int kk_{t}ds-\frac{1}{2\pi}\int
k^{3}(k-\alpha)ds\\
&=&\frac{1}{\pi}\int
k(k_{ss}+k^{2}(k-\alpha))ds-\frac{1}{2\pi}\int k^{3}(k-\alpha)ds\\
&=&-\frac{1}{\pi}\int (k_{s})^{2}ds+\frac{1}{2\pi}\int
k^{3}(k-\alpha)ds.
\end{eqnarray*}

\end{proof}

In principle we may consider the evolution of $f:=\frac{k}{P-\rho}$
for $\rho$ is a circle inside the region bounded by $\gamma(t)$ and
get control of the curvature of the evolving curve. However, unlike
the shrinking flow where $\alpha=0$, we can not use the maximum
principle to get the bound of $f$ to show the convergence of the
flow. So we need to use entropy estimate and integral estimate as
done by Gage and Hamilton for the shrinking curve flow \cite{GH}.
These will be done in below.

\section{Long time existence}\label{sect2}
In this section, we derive key estimates of the curvature $k$ of the
evolving curve $\gamma(t)$.

In the following, we work with the general curve flow
\begin{equation}\label{eq}
    \frac{\partial}{\partial t}\gamma=(k-\alpha)N+\eta T
\end{equation}
where $\eta$ will be determined later.

Similar to above, we have basic formulae
\begin{lem}
$$
\partial_{t}\partial_{s}-\partial_{s}\partial_{t}=k(k-\alpha)\partial_{s}-\eta_{s}\partial_{s}.
$$

$$
\partial_{t}T=(\partial_{s}(k-\alpha)+k\eta)N.
$$

$$
\partial_{t}\theta=\partial_{s}k+k\eta.
$$

$$
\partial_{t}L=0.
$$

$$
\partial_{t}A=\int (\alpha-k)ds.
$$
\end{lem}

We take $\eta$ such that $\partial_{t}\theta=0$, i.e.
$\eta=-\partial_{\theta}k$.

\begin{rem}
Under this flow, the length is also preserved, and the variation of
area is the same as $\eta=0$.
\end{rem}

\begin{lem}
$$
\frac{\partial k}{\partial
 t}=k^{2}\partial_{\theta}^{2}k+k^{2}(k-\alpha).
$$
\end{lem}
\begin{proof}
\begin{eqnarray*}
\partial_{t}k&=&\partial_{t}\partial_{s}\theta=\partial_{s}\partial_{t}\theta+k(k-\alpha)\partial_{s}\theta-\partial_{s}\eta
k\\
&=&k^{2}\partial_{\theta}^{2}k+k^{2}(k-\alpha).
\end{eqnarray*}
\end{proof}

\begin{thm}
Convexity is preserved along the flow (\ref{eq}). In fact, for any
finite time $T<\infty$, such that the curve flow exists on $[0,T]$,
we have $k(t)$ is uniformly bounded from below on $t\in [0,T]$
\end{thm}

\begin{proof}
By direct computation, we have
$$
\partial_{t}(\frac{1}{k}-\frac{A}{L}-\frac{2\pi
t}{L})=k^{2}\partial_{\theta}^{2}(\frac{1}{k}-\frac{A}{L}-\frac{2\pi
t}{L})-2k^{3}(\partial_{\theta}(\frac{1}{k}-\frac{A}{L}-\frac{2\pi
t}{L}))^{2}-k.
$$

Before $\frac{1}{k}-\frac{A}{L}-\frac{2\pi t}{L}$ blow up, we have
$k>0.$

Since the curve is compact, we can take the maximum of
$\frac{1}{k}-\frac{A}{L}-\frac{2\pi t}{L}$ at $(x_{0},t_{0})\in
\gamma\times [0,T]$, where the curve flow exists on $[0,T]$. If
$t_{0}>0$, we have
$$
\frac{\partial}{\partial t}(\frac{1}{k}-\frac{A}{L}-\frac{2\pi
t}{L})|_{(x_{0},t_{0})}\geq 0.
$$
But
$$
k^{2}\partial_{\theta}^{2}(\frac{1}{k}-\frac{A}{L}-\frac{2\pi
t}{L})|_{x_{0},t_{0}}\leq
0,~2k^{3}(\partial_{\theta}(\frac{1}{k}-\frac{A}{L}-\frac{2\pi
t}{L}))^{2}|_{x_{0},t_{0}}=0,~-k<0
$$
A contradiction. So
$$
\frac{1}{k}-\frac{A}{L}-\frac{2\pi t}{L}\leq
max(\frac{1}{k(0)})-\frac{A(0)}{L}
$$
That is $k(t)\geq \frac{1}{C_{1}+C_{2}t}=c(t)$, where $C_{1}$ and
$C_{2}$ are positive constant. So the convexity is preserved along
the flow.
\end{proof}

\begin{thm}
$\int log k(\theta,t)d\theta$ is non increasing along the flow, So
there is a uniform bound of $\int log k(\theta,t)d\theta$ along the
flow.
\end{thm}
\begin{proof}
\begin{eqnarray*}
\frac{\partial}{\partial t}\int log k(\theta,t)d\theta &=& \int
k\partial_{\theta}^{2}k+k(k-\alpha)d\theta\\
&=&\int_{0}^{2\pi}-(\partial_{\theta}k)^{2}
+(k-\alpha)^{2}d\theta+\alpha\int_{0}^{2\pi}(k-\alpha)d\theta.
\end{eqnarray*}

By definition, we have $\int_{0}^{2\pi}(k-\alpha)d\theta=0.$ Using
the Wirtinger inequality, we have
$$
\frac{\partial}{\partial t}\int_{0}^{2\pi}log k(\theta,t)d\theta\leq
0.
$$

Hence,
$$
\int_{0}^{2\pi}log k(\theta,t)d\theta\leq \int_{0}^{2\pi}log
k(\theta,0)d\theta
$$
\end{proof}

\begin{thm} \cite{GH}
Suppose the curve flow exists on $[0,T)$. For any $\delta>0,$ we can
find a  constant $C(T)$ such that $k(\theta,t)\leq C(T)$ except on
intervals of length less than or equal to $\delta$.
\end{thm}
\begin{proof}
If $k\geq C(t)$ on $a\leq \theta\leq b$ and $b-a\geq \delta$, then
\begin{eqnarray*}
\int_{0}^{2\pi}log k(\theta,t)d\theta&\geq& \delta
logC(T)+(2\pi-\delta)log k_{min}(t)\\
&\geq& \delta logC(T)+(2\pi-\delta)log c(T)
\end{eqnarray*}
where $c(T)$ is the low bound of $k(t)$ on $[0,T)$. Since
$\int_{0}^{2\pi}log k(\theta,t)d\theta$ is non increasing, $C(T)$ is
bounded above.
\end{proof}

\begin{lem}
We have
$$
 \int (\frac{\partial k}{\partial \theta})^{2}d\theta \leq
\int k^{2}+D
$$
where $D$ is a constant which depends only on $\gamma(0)$.
\end{lem}
\begin{proof}
\begin{eqnarray*}
& &\frac{\partial}{\partial t}\int (k-\alpha)^{2}-(\frac{\partial
k}{\partial \theta})^{2}d\theta\\
&=&2\int(k-\alpha)(\partial_{t}k-\partial_{t}\alpha)-2\frac{\partial
k}{\partial \theta}\frac{\partial^{2}k}{\partial \theta \partial
t}d\theta\\
&=&2\int (k-\alpha+\partial_{\theta}^{2}k)\partial_{t}kd\theta-2\int
(k-\alpha)\partial_{t}\alpha d\theta\\
&=&2\int
(k-\alpha+\partial_{\theta}^{2}k)^{2}k^{2}d\theta+2\partial_{t}\alpha(2\pi\alpha-\int
k d\theta)
\end{eqnarray*}
But
$$
\int kd\theta=\int k^{2}ds=2\pi\alpha.
$$
So
$$
\frac{\partial}{\partial t}\int (k-\alpha)^{2}-(\frac{\partial
k}{\partial \theta})^{2}d\theta=2\int
(k-\alpha+\partial_{\theta}^{2}k)^{2}k^{2}d\theta
$$
Integrating the above inequality, we have
$$
\int (k(t)-\alpha(t))^2-(\frac{\partial k(t)}{\partial
\theta})^{2}d\theta\geq\int (k(0)-\alpha(0))^2-(\frac{\partial
k(0)}{\partial \theta})^{2}d\theta=-D.
$$
Then we have
$$
\int (\frac{\partial k(t)}{\partial \theta}^{2})d\theta\leq \int
(k(t)-\alpha(t))^{2}+D\leq \int k^{2}d\theta +D.
$$
\end{proof}

\begin{thm}
If $\int_{0}^{2\pi}log k(\theta,t)d\theta$ is bounded on $[0,T)$,
then $k(\theta,t)$ is uniformly bounded on $\gamma\times [0,T)$.
\end{thm}
\begin{proof}
For any given $\delta$, by the above estimate, we have $k\leq C(T)$
except on intervals $[a,b]$ of length less than $\delta$. On such an
interval
\begin{eqnarray*}
k(\phi)=k(a)+\int_{a}^{\phi}\frac{\partial k}{\partial
\theta}d\theta\leq C(T)+\sqrt{\delta}(\int (\frac{\partial
k}{\partial \theta})^{2}d\theta)^{1/2}\leq C(T)+\sqrt{\delta}(\int
k^{2}d\theta+D)^{1/2}.
\end{eqnarray*}
This shows that if $k_{max}$ is the maximum value of $k$,then
$$
k_{max}\leq C(T)+\sqrt{\delta}(2\pi k_{max}^{2}+D)^{1/2}.
$$
By choosing $\delta$ small, we have
$$
k_{max}^{2}\leq \frac{2C^{2}(T)+2\delta D}{1-4\pi\delta}\leq
4C^{2}(T).
$$
\end{proof}

\begin{lem}
If k is bounded, then $\frac{\partial k}{\partial \theta}$ is
bounded.
\end{lem}

\begin{proof}
$$
\partial_{t}\partial_{\theta}k=k^{2}\partial_{\theta}^{3}k+2k\partial_{\theta}k\partial_{\theta}^{2}k+3k^{2}\partial_{\theta}k-2\alpha
k\partial_{\theta}k.
$$
Since $k$ is bounded, $\alpha$ is bounded. So $\partial_{\theta}k$
grows at most exponentially.
\end{proof}

In the following, we will use $k^{'}$, etc, to denote the
derivatives of $\gamma(t)$ w.r.t variable $\theta$,
$\partial_{\theta}k$, etc.

\begin{lem}
If $k$ and $k^{'}$ are bounded, then
$\int_{0}^{2\pi}(k^{''})^{4}d\theta$ is bounded.
\end{lem}

\begin{proof}
\begin{eqnarray*}
\frac{\partial}{\partial
t}\int_{0}^{2\pi}(k^{''})^{4}d\theta&=&4\int_{0}^{2\pi}(k^{''})^{3}(k^{2}k^{''}+k^{2}(k-\alpha))^{''}d\theta\\
&=&-12\int_{0}^{2\pi}(k^{''})^{2}(k^{'''})(k^{2}k^{'''}+2kk^{'}k^{''}+3k^{2}k^{'}-2\alpha
kk^{'})d\theta\\
&=&-12\int_{0}^{2\pi}
k^{2}(k^{''})^{2}(k^{'''})^{2}+2kk^{'}(k^{''})^{3}k^{'''}+3k^{2}k^{'}(k^{''})^{2}k^{'''}d\theta\\
& &+24\alpha\int_{0}^{2\pi}kk^{'}(k^{''})^{2}k^{'''}d\theta\\
&\leq&
C_{1}\int_{0}^{2\pi}(k^{''})^{4}(k^{'})^{2}d\theta+C_{2}\int_{0}^{2\pi}k^{2}(k^{'})^{2}(k^{''})^{2}d\theta\\
& & +24\alpha C_{3}\int_{0}^{2\pi}(k^{'})^{2}(k^{''})^{2}d\theta
\end{eqnarray*}
By the bound of $k,k^{'}$, we see that $\int_{0}^{2\pi}(k^{''})^{4}$
grows at most exponentially.
\end{proof}

\begin{lem}
If $k,k^{'}$ and $\int_{0}^{2\pi} (k^{''})^{4}d\theta$ are bounded,
so is $\int_{0}^{2\pi}(k^{'''})^{2}d\theta$.
\end{lem}
\begin{proof}
\begin{eqnarray*}
& &\frac{\partial}{\partial
\theta}\int_{0}^{2\pi}(k^{'''})^{2}d\theta\\
&=&-2\int_{0}^{2\pi}k^{''''}(k^{2}k^{''}+k^{3}-\alpha k^{2})^{''}d\theta\\
&=&-2\int_{0}^{2\pi}k^{2}(k^{''''})^{2}+4kk^{'}k^{'''}k^{''''}+2k(k^{''})^{2}k^{''''}\\
& &+2(k^{'})^{2}k^{''}k^{''''}
+3k^{2}k^{''}k^{''''}+6k(k^{'})^{2}k^{''''}\\
& &-2\alpha
(k^{'})^{2}k^{''''}-2\alpha kk^{''}k^{''''}d\theta\\
& \leq & C_{1}\int (k^{'})^{2}(k^{'''})^{2}d\theta+C_{2}\int
(k^{''})^{4}d\theta+C_{3}\int
\frac{(k^{'})^{4}}{k^{2}}(k^{''})^{2}d\theta\\
&+&C_{4}\int k^{2}(k^{''})^{2}d\theta+C_{5}\int
(k^{'})^{4}d\theta+C_{6}\int
\frac{(k^{'})^{4}}{k^{2}}d\theta+C_{7}\int (k^{''})^{2}d\theta.
\end{eqnarray*}
By the bounds of $k,~k^{'},~||k^{''}||_{4}$, we see that
$\int_{0}^{2\pi}(k^{'''})^{2}d\theta$ grows at most exponentially.
\end{proof}

\begin{cor}
Under the same hypothesis, $k^{''}$ is bounded.
\end{cor}
\begin{proof}
In one dimension
$$
max|f|^{2}\leq C\int |f^{'}|^{2}+f^{2}
$$
We apply this to $k^{''}$.
\end{proof}

\begin{lem}
If $k,k^{'}$, and $k^{''}$ are uniformly bounded, then so are
$k^{'''}$ and all the higher derivatives.
\end{lem}

\begin{proof}
We compute
\begin{eqnarray*}
\frac{\partial}{\partial t}k^{'''}&=&(k^{2}k^{''}+k^{3}-k^{2}\alpha)^{'''}\\
&=&k^{2}k^{v}+6kk^{'}k^{iv}+(8kk^{''}+6k^{'2}+3k^{2}-2\alpha
k)k^{'''}\\
& &+(6k^{'}(k^{''})^{2}+18kk^{'}k^{''}+6(k^{'})^{3}-6\alpha
k^{'}k^{''})
\end{eqnarray*}
Since $k,k^{'},k^{''},\alpha$ is bounded, $k^{'''}$ grows at most
exponentially. Similarly, we can show that $k^{(n)}$ is bounded on
finite intervals.
\end{proof}
From the above analysis, we have
\begin{thm}
The curve flow does not blow up in finite time.
\end{thm}

\begin{thm}\cite{G}\label{isoperi}
For the closed, convex $C^{2}$ curve $\gamma(t)$ in the plane, we
have
$$
\pi\frac{L}{A}\leq \int_{0}^{L}k^{2}ds,
$$
where $L,A$ and $k$ are the length of the curve, the area it
encloses, and its curvature respectively.
\end{thm}

\begin{thm}
If the convex curve $\gamma(t)$ evolves according to (\ref{eq}),
then the isoperimetric deficit $L^{2}-4\pi A$ is decreasing during
the evolution process (\ref{cf}) and converges to zero as the time
$t$ goes to infinity.
\end{thm}

\begin{proof}
Since the length of the curve is preserved,
$$
\frac{d}{dt}(L^{2}-4\pi
A)=-4\pi\frac{d}{dt}A(t)=-4\pi(L\alpha-2\pi)\leq 0.
$$
From the above theorem \ref{isoperi}, we have
$$
\frac{d}{dt}(L^{2}-4\pi A)\leq
-4\pi(\frac{L^{2}}{2A}-2\pi)=-\frac{2\pi}{A}(L^{2}-4\pi A),
$$
Since for any plane curve,
$$
\frac{L^{2}}{4\pi}\geq A,
$$
we have
$$
\frac{d}{dt}(L^{2}-4\pi A)\leq -\frac{8\pi^{2}}{L^{2}}(L^{2}-4\pi A)
$$
So
$$
L^{2}-4\pi A(t)\leq Cexp(-\frac{8\pi^{2}}{L^{2}}t).
$$
As $t\rightarrow \infty,$ we have
$$
L^{2}-4\pi A\rightarrow 0.
$$

By Bonnesen inequality (see \cite{OS}), $ \frac{L^{2}}{A}-4\pi\geq
\frac{\pi^{2}}{A}(r_{out}-r_{in})^{2}$, we have
$r_{out}-r_{in}\rightarrow 0$, as $t\rightarrow \infty$. So the
curve converge to a circle in Hausdorff sense.
\end{proof}

\section{$C^{\infty}$ convergence}\label{sect3}
We will follow Hamilton's estimate to prove the $C^{\infty}$
convergence.

Define
$$
k^{*}_{w}=sup\{b|k(\theta)>b ~on ~some ~interval ~of ~length ~w\}.
$$

\begin{lem}\cite{GH}\label{fe}
$$
k^{*}_{w}(t)r_{in}(t)\leq
\frac{1}{1-K(w)(\frac{r_{out}}{r_{in}}-1)},
$$
where $r_{in}$ and $r_{out}$ are  the radii of the largest inscribed
circle and the smallest circumscribed circle of the curve defined by
the $k(\dot,t)$ respectively, and $K$ is a positive decreasing
function of $w$ with $K(0)=\infty$ and $K(\pi)=0$.
\end{lem}

\begin{rem}
Here we need the explicit formula of $K(w).$
$$
K(w)=\frac{2cos(\frac{w}{2})}{1-cos(\frac{w}{2})}.
$$ See also \cite{GH}.
\end{rem}

By the above section, we have the curve $\gamma(t)$ converges to a
circle in Hausdorff sense as $t\to\infty$, i.e.
$$
\pi(r_{out}-r_{in})^{2}\leq L^{2}-4\pi A \rightarrow 0.
$$

We also have $\pi r_{out}^{2}\geq A(t)\geq A(0)$, and
$r_{out}\rightarrow r_{in}$, So There is a sufficient large time
$T_{1}$ such that $r_{in}(t)\geq \sqrt{\frac{A_{0}}{2\pi}}$ for
$t\geq T_{1}$.

We first fix a small $w$. Then there is a sufficient large time
$T_{2}\geq T_{1}$, such that
$$
K(w)(\frac{r_{out}}{r_{in}}-1)\leq 1/2.
$$
for $t\geq T_{2}$. By the theorem \ref{fe}, we have
$
k_{w}^{*}(t)r_{in}(t)\leq 2,
$
i.e $~k_{w}^{*}(t)\leq 2\sqrt{\frac{2\pi}{A_{0}}}$ for $t\geq
T_{2}$.

Then by a use of the method used by Gage and Hamilton in \cite{GH},
we have
\begin{thm}
Curvature $k(t)$ uniformly bounded along the flow.
\end{thm}
\begin{proof}
We just need to consider the curvature $k$ of the curve $\gamma(t)$
for $t\geq T_{2}$. First we fix a small $w<(\frac{1}{4\pi})^{2}$.
Suppose that $[a,b]$ is a interval such that $k\geq k_{w}^{*}$. By
condition, we have $|b-a|\leq w$ and $k(a)=k_{w}^{*}$. For any
$\phi\in [a,b]$, we have
\begin{eqnarray*}
k(\phi)&=& k(a)+\int_{a}^{\phi}\frac{\partial k}{\partial
\theta}d\theta\leq k_{w}^{*}+\sqrt{w}(\int (\frac{\partial
k}{\partial \theta})^{2}d\theta)^{1/2}\\
& &\leq k_{w}^{*}+\sqrt{w}(\int k^{2}d\theta+D)^{1/2}.
\end{eqnarray*}
If $k_{max}$ is the maximum value of $k$, then
$$
k_{max}\leq k_{w}^{*}+\sqrt{w}(2\pi k_{max}^{2}+D)^{1/2}\leq
k_{w}^{*}+2\pi \sqrt{w}k_{max}+\sqrt{w}D.
$$
So for $t\geq T_{2}$, we have $k$ is bounded uniformly.
\end{proof}

Since $k(\theta,t)$ is uniformly bounded,  $\int
(\partial_{\theta}k)^{2}d\theta$ is also uniformly bounded. Then we
have $k(\theta,t)$ is equi-continuous. So for any sequence
$k(\theta,t_{i})$, we can choose a sequence $k(\theta,t_{i_{n}})$
converge uniformly to $k(\theta,\infty)$. But we know the curve
converge to circle in the Hausdorff sense. So
$k(\theta,\infty)=const$. Since every subsequences converge to
$k(\theta,\infty)=const$, we have $k(\theta,t)$ converge to
$k(\theta,\infty)=const$ uniformly.

Similar to \cite{GH}, we have
\begin{lem}\label{k4}
$||k^{'}||_{4}$ are bounded by constants independent of $t$.
\end{lem}
\begin{proof}
\begin{eqnarray*}
\frac{\partial}{\partial t}\int (k^{'})^{4}d\theta&=&4\int
(k^{'})^{3}(k^{2}k^{''}+k^{2}(k-\alpha))^{'}d\theta\\
&=&-12\int_{0}^{2\pi}k^{2}(k^{'})^{2}(k^{''})^{2}
-12\int_{0}^{2\pi}(k^{'})^{2}k^{''}
k^{3}d\theta\\
& &-8\alpha\int_{0}^{2\pi}(k^{'})^{4}kd\theta\\
& &\leq 3\int_{0}^{2\pi}
k^{4}(k^{'})^{2}d\theta-8\alpha\int_{0}^{2\pi}(k^{'})^{4}d\theta
\end{eqnarray*}
Since $k$ converges to a constant, $\alpha(t)$ is converging to a
constant. Using the Holder inequality, we have
$$
\frac{\partial f}{\partial t}\leq C_{1}f^{1/2}-C_{2}f
$$
where $f=\int_{0}^{2\pi}(k^{'})^{4}d\theta$, $C_{1}$ and $C_{2}$ is
independent of $t$. By the lemma 5.7.4 of \cite{GH}, we have
$||k^{'}||_{4}(t)$ is uniformly bounded.
\end{proof}

\begin{lem}
$||k^{''}||_{2}$ is bounded by a constant which is independent of
$t$.
\end{lem}
\begin{proof}
\begin{eqnarray*}
& &\partial_{t}\int_{0}^{2\pi}(k^{''})^{2}d\theta\\
&=&\int_{0}^{2\pi}k^{''}(\partial_{t}k)^{''}d\theta\\
&=&-\int_{0}^{2\pi}k^{'''}(k^{2}k^{''})^{'}d\theta+\int_{0}^{2\pi}k^{''}((k^{3})^{''}-\alpha(k^{2})^{''})d\theta\\
&=&-\int_{0}^{2\pi}k^{2}(k^{'''})^{2}d\theta-2\int_{0}^{2\pi}kk^{'}k^{''}k^{'''}d\theta\\
&
&-3\int_{0}^{2\pi}k^{2}k^{'}k^{'''}d\theta-2\alpha\int_{0}^{2\pi}k(k^{''})^{2}d\theta
-2\alpha\int_{0}^{2\pi}k^{''}(k^{'})^{2}d\theta
\end{eqnarray*}
By the Cauchy inequality, we have
\begin{eqnarray*}
\partial_{t}\int_{0}^{2\pi}(k^{''})^{2}d\theta &\leq&
C_{1}\int_{0}^{2\pi}(k^{'}k^{''})^{2}d\theta+C_{2}\int_{0}^{2\pi}k^{2}(k^{'})^{2}d\theta
-2\alpha\int_{0}^{2\pi}k(k^{''})^{2}d\theta\\
& &-2\alpha\int_{0}^{2\pi}k^{''}(k^{'})^{2}d\theta
\end{eqnarray*}
We can control the first term by the inequality in above lemma
\ref{k4},
\begin{eqnarray*}
\partial_{t}\int_{0}^{2\pi}(k^{''})^{2}d\theta &\leq& C_{2}\int_{0}^{2\pi}k^{2}(k^{'})^{2}d\theta
-2\alpha\int_{0}^{2\pi}k(k^{''})^{2}d\theta\\
&
&-2\alpha\int_{0}^{2\pi}k^{''}(k^{'})^{2}d\theta-C_{3}\partial_{t}\int_{0}^{2\pi}(k^{'})^{4}d\theta\\
&
&-C_{4}\int_{0}^{2\pi}(k^{'})^{2}k^{''}k^{3}d\theta-C_{5}\int_{0}^{2\pi}(k^{'})^{4}d\theta\\
& &\leq
C_{6}-C_{7}\int_{0}^{2\pi}(k^{''})^{2}d\theta-C_{3}\partial_{t}\int_{0}^{2\pi}(k^{'})^{4}d\theta
\end{eqnarray*}
We denote $\int_{0}^{2\pi}(k^{''})^{2}d\theta$ by $f(t)$. Then we
have
$$
\partial_{t}f\leq C_{6}-C_{7}f-C_{3}\partial_{t}\int_{0}^{2\pi}(k^{'})^{4}d\theta
$$
Multiplying $e^{c_{7}t}$ on both side and integrating, we have
$$
e^{c_{7}t}f(t)|_{0}^{t}\leq
C_{6}\int_{0}^{t}e^{C_{7}t}dt-C_{3}\int_{0}^{t}e^{C_{7}t}\partial_{t}\int_{0}^{2\pi}(k^{'})^{4}d\theta
dt
$$
So we have
\begin{eqnarray*}
e^{C_{7}t}f(t)& \leq &
\frac{C_{6}}{C_{7}}e^{C_{7}t}+C_{3}C_{7}\int_{0}^{t}e^{C_{7}t}\int_{0}^{2\pi}(k^{'})^{4}d\theta+C_{8}\\
& \leq &\frac{C_{6}}{C_{7}}e^{C_{7}t}+C_{3}Me^{C_{7}t}+C_{8}.
\end{eqnarray*}
where $M$ is the bound of $\int_{0}^{2\pi}(k^{'})^{4}d\theta$ which
is independent of $t$. So $\int_{0}^{2\pi}(k^{''})^{2}d\theta$ is
uniformly bounded.
\end{proof}
As to lemma 5.7.8 of \cite{GH}, we have
\begin{lem}
$||k^{'}||_{\infty}$ converges to 0 as $t\rightarrow \infty$.
\end{lem}

\begin{lem} \cite{GH}
For any $0<\beta<1$ we can choose $A$ so that for $t>A$
$$
\int (k^{''})^{2}d\theta\geq 4\beta\int (k^{'})^{2}d\theta.
$$
\end{lem}

\begin{lem}
For any $0<\beta<1$, there is a constant $C_{1}$ such that
$||k^{'}||_{2}\leq C_{1}e^{-2\beta C^{2} t}$, where $k\rightarrow
C$, as $t\rightarrow \infty$.
\end{lem}
\begin{proof}
\begin{eqnarray*}
\partial_{t}\int (k^{'})^{2}d\theta&=&2\int k^{'}(k^{2}k^{''}+k^{3}-\alpha
k^{2})^{'}d\theta\\
&=&-2\int k^{2}(k^{''})^{2}d\theta+6\int
k^{2}(k^{'})^{2}d\theta-4\alpha\int k(k^{'})^{2}d\theta.
\end{eqnarray*}
Since $k\rightarrow C,$ as $t\rightarrow \infty$, we have
$$
\partial_{t}\int (k^{'})^{2}d\theta\leq -2\beta C^{2}\int
(k^{'})^{2}d\theta.
$$
So we complete the lemma.
\end{proof}

\begin{lem}
We can find a constant $C$ such that $||k^{''}||_{2}\leq Ce^{-2\beta
t}$, where $\beta>0$ is a uniform constant .
\end{lem}
\begin{proof}
By direct computation, we have
\begin{eqnarray*}
& &\partial_{t}\int (k^{''})^{2}d\theta\\
&=&2\int
k^{''}(k_{t})^{''}d\theta=2\int
k^{''}(k^{2}k^{''}+k^{2}(k-\alpha))^{''}d\theta\\
&=&-2\int k^{2}(k^{'''})^{2}d\theta-4\int
kk^{'}k^{''}k^{'''}d\theta-6\int
k^{2}k^{'}k^{'''}d\theta\\
& &-4\alpha\int (k^{'})^{2}k^{''}-4\alpha\int k(k^{''})^{2}d\theta\\
& \leq &-2\int k^{2}(k^{'''})^{2}d\theta+4\epsilon\int
k^{2}(k^{'''})^{2}d\theta+1/\epsilon\int
(k^{'})^{2}(k^{''})^{2}d\theta\\
& &+6(\epsilon\int (kk^{'''})^{2}+1/4\epsilon\int
k^{2}(k^{'})^{2})-4\alpha\int (k^{'})^{2}k^{''}d\theta\\
& &-4\alpha\int k(k^{''})^{2}d\theta
\end{eqnarray*}
First we choose $\epsilon>0$ small such that
\begin{eqnarray*}
\partial_{t}\int (k^{''})^{2}d\theta& \leq &C_{1}\int
(k^{'})^{2}(k^{''})^{2}d\theta+C_{2}\int
k^{2}(k^{'})^{2}d\theta\\
& &-4\alpha\int (k^{'})^{2}k^{''}d\theta-4\alpha\int
k(k^{''})^{2}d\theta
\end{eqnarray*}
Since $||k^{'}||_{\infty}$ converges to $0$, $k\rightarrow C$,
$\alpha\rightarrow C$ as $t\rightarrow \infty$, we have
$$
\partial_{t}\int (k^{''})^{2}d\theta\leq -2C^{2}\int
(k^{''})^{2}d\theta+C_{2}Ce^{-C_{3}t}.
$$
We denote $\int (k^{''})^{2}d\theta$ by $f$. Then we have
$$
\partial_{t}f\leq -2C^{2}f+C_{2}Ce^{-C_{3}t}.
$$
By the lemma 5.7.6 of \cite{GH}, we complete the proof.

\end{proof}

\begin{lem}
We can find uniform constants $C$ and $\beta>0$ such that
$||k^{''}||_{4}\leq Ce^{-\beta t}$.
\end{lem}
\begin{proof}
\begin{eqnarray*}
& &\partial_{t}\int (k^{''})^{4}d\theta\\&=&4\int
(k^{''})^{3}(k^{2}k^{''}+k^{2}(k-\alpha))^{''}d\theta\\
&=&-12\int k^{2}(k^{''})^{2}(k^{'''})^{2}d\theta-24\int k
k^{'}(k^{''})^{3}k^{'''}d\theta\\
& &-36\int k^{2}k^{'}(k^{''})^{2}k^{'''}d\theta-8\alpha\int
(k^{'})^{2}(k^{''})^{3}d\theta-8\alpha\int
k(k^{''})^{4}d\theta\\
&\leq &-12\int k^{2}(k^{''})^{2}(k^{'''})^{2}d\theta+24(\epsilon\int
k^{2}(k^{''})^{2}(k^{'''})^{2}+1/4\epsilon\int
(k^{'})^{2}(k^{''})^{4}d\theta)\\
& &+36(\epsilon\int k^{2}(k^{''})^{2}(k^{'''})^{2}+1/4\epsilon\int
k^{2}(k^{'})^{2}(k^{''})^{2}d\theta)\\
& &-8\alpha\int (k^{'})^{2}(k^{''})^{3}d\theta-8\alpha\int
k(k^{''})^{4}d\theta\\
& \leq &C_{1}\int (k^{'})^{2}(k^{''})^{4}d\theta+C_{2}\int
k^{2}(k^{'})^{2}(k^{''})^{2}d\theta\\
& & -8\alpha\int (k^{'})^{2}(k^{''})^{3}d\theta-8\alpha\int
k(k^{''})^{4}d\theta.
\end{eqnarray*}
By Young inequality, we have
$$
\int (k^{'})^{2}(k^{''})^{3}d\theta\leq \epsilon \int
(k^{'})^{4/3}(k^{''})^{4}d\theta+C(\epsilon)\int (k^{'})^{4}d\theta,
$$
$$
\int k^{2}(k^{'})^{2}(k^{''})^{2}d\theta\leq \epsilon\int
(k^{''})^{4}d\theta+C(\epsilon)\int k^{4}(k^{'})^{4}d\theta.
$$
Since $||k^{'}||_{\infty}\rightarrow 0$ as $t\rightarrow \infty$, we
have
$$
\partial_{t}\int (k^{''})^{4}d\theta\leq -C_{1}\int
(k^{''})^{4}d\theta+C_{2}\int (k^{'})^{4}d\theta.
$$
Denote $\int (k^{''})^{4}d\theta$ by $f$. We have
$$
\partial_{t}f\leq -C_{1}f+C_{2}e^{-C_{3}t}
$$
By lemma 5.7.6 in \cite{GH}, we complete the proof.
\end{proof}

\begin{lem}
There is some constant $C$ such that $||k^{'''}||_{2}\leq
Ce^{-2\beta t}$, where $\beta>0$ is a uniform constant.
\end{lem}
\begin{proof}
\begin{eqnarray*}
\partial_{t}\int (k^{'''})^{2}d\theta&=&-2\int
k^{2}(k^{''''})^{2}d\theta-8\int kk^{'}k^{'''}k^{''''}d\theta\\
& &-4\int (k^{'})^{2}k^{''}k^{''''}d\theta-4\int
k(k^{''})^{2}k^{''''}d\theta-6\int
k^{2}k^{''}k^{''''}d\theta\\
& &-12\int k(k^{'})^{2}k^{''''}d\theta-4\alpha\int
k(k^{'''})^{2}d\theta-12\alpha\int k^{'}k^{''}k^{'''}d\theta.
\end{eqnarray*}
By Young inequality, we have
\begin{eqnarray*}
\partial_{t}\int (k^{'''})^{2}d\theta&\leq &-2\int
k^{2}(k^{''''})^{2}d\theta+\epsilon \int
k^{2}(k^{''''})^{2}d\theta+C_{1}(\epsilon)\int
(k^{'})^{2}(k^{'''})^{2}d\theta\\
& &+C_{2}(\epsilon)\int
\frac{(k^{'})^{4}(k^{''})^{2}}{k^{2}}d\theta+C_{3}(\epsilon)\int
(k^{''})^{4}d\theta\\
& &+C_{4}(\epsilon)\int (k^{'})^{4}d\theta+C_{5}(\epsilon)\int
k^{2}(k^{''})^{2}d\theta-4\alpha\int
k(k^{'''})^{2}d\theta\\
& &-12\alpha(\epsilon\int (k^{'''})^{2}d\theta+C_{6}(\epsilon)\int
(k^{'})^{2}(k^{''})^{2}d\theta).
\end{eqnarray*}
By the above estimate of $k^{'}$ and $k^{''}$, for sufficiently
large $t$, we have
$$
\partial_{t}\int (k^{'''})^{2}d\theta\leq -C_{6}\int
(k^{'''})^{2}d\theta+C_{7}e^{-\beta t}
$$
The result follows from the lemma 5.7.5 in \cite{GH}.
\end{proof}

In fact, the method in \cite{GH} can be applied to our case. The
high order estimate of $k$ is similar to \cite{GH}, so we omit the
detail.

\end{document}